\documentclass[leqno, 12pt, a4paper]{amsart}

\usepackage{amssymb}
\usepackage{amsthm}
\usepackage{amsmath}
\usepackage{amsfonts}
\usepackage{amscd}
\usepackage{array}
\usepackage{float}
\usepackage{cite}
\usepackage{showlabels}
\usepackage{color}
\usepackage{enumerate}
\usepackage{hyperref}

\numberwithin{equation}{section}

\theoremstyle{definition}
\newtheorem{defn}{Definition}[section]

\newtheorem{ex}[defn]{Example}

\newtheorem{ques}[defn]{Question}

\theoremstyle{plain}
\newtheorem{prop}[defn]{Proposition}
\newtheorem{lem}[defn]{Lemma}

\newtheorem{cor}[defn]{Corollary}

\theoremstyle{remark}
\newtheorem{rem}[defn]{Remark}

\newcommand{\bA}{\mathbb{A}}

\newcommand{\bF}{\mathbb{F}}

\newcommand{\bP}{\mathbb{P}}
\newcommand{\bQ}{\mathbb{Q}}

\newcommand{\bZ}{\mathbb{Z}}

\newcommand{\cB}{\mathcal{B}}

\newcommand{\cF}{\mathcal{F}}

\newcommand{\cH}{\mathcal{H}}

\newcommand{\cL}{\mathcal{L}}
\newcommand{\cM}{\mathcal{M}}
\newcommand{\cN}{\mathcal{N}}
\newcommand{\cO}{\mathcal{O}}

\newcommand{\cT}{\mathcal{T}}

\newcommand{\codim}{\operatorname{codim}}

\newcommand{\vol}{\operatorname{Vol}}

\newcommand{\lra}{\longrightarrow}

\def\ni{\noindent}

\address[]{Department of Mathematics, National Taiwan University, Taipei 10617, Taiwan
}

\address[]{National Center for Theoretical Sciences, Mathematics Division, Taipei 10617, Taiwan
}
\email{jkchen@ntu.edu.tw}
\address[]{Department of Mathematics, National Taiwan University, Taipei 10617, Taiwan (R.O.C.)
}
\email{cjlai72@math.ntu.edu.tw}
\begin{document}
\title{Varieties of general type with small volumes}
\author{Jungkai Alfred Chen, Ching-Jui Lai}
\date{\today}
\maketitle

\begin{abstract}
Generalize Kobayashi's example for the Noether inequality in dimension three, we provide examples of $n$-folds of general type with small volumes.
\end{abstract}


\section{Introduction}

From the point of view of classification theory, the following three type of algebraic varieties are considered to be the building blocks: varieties of Fano-type, varieties with $\kappa=0$, and varieties of general type. Unlike the other two types, varieties of general type can not be bounded. For example, curves of general type are those curves of genus $g\ge 2$, where $g$ can be arbitrarily large.  Nevertheless, the classical theory of surfaces shows that there are some relations between some fundamental invariants. The Bogomolov-Miyaoka-Yau inequality shows that $c_1^2 \le 3c_2$ and Noether inequality gives $K^2 \ge 2p_g -4$.


It is thus natural and interesting to study the geography of higher dimensional varieties of general type, including the distribution of birational invariants and relations among invariants.
There are some recent results in dimension three, due to Kobayashi, Meng Chen and many others. Let $d_1$ denotes the dimension of image of the canonical map $\varphi_{|K_X|} \colon X \dashrightarrow \bP^{p_g(X)-1}$. In \cite{K}, Kobayashi proved that for an $n$-dimensional smooth projective variety of general type, $\vol(X) \ge 2 (p_g(X)-n)$ provided $d_1=\dim X$. 　
Hence in particular, for threefolds of general type with $d_1=3$, it follows that $\vol(X) \ge 2 (p_g(X)-3)$. However, Kobayashi constructed examples of threefolds of general type with $d_1=2$ such that $\vol(X)= \frac{4}{3} p_g(X) - \frac{10}{3}$. A recent result of Meng Chen and the first-named authors prove that $\vol(X) \ge  \frac{4}{3} p_g(X) - \frac{10}{3}$ for any threefold of general type whose minimal model is Gorenstein, cf. \cite{CC}.

The purpose of this note is to explore the higher dimensional analogue of Noether inequality.
We would like to remark the essential difference and difficulty in dimension three or higher is that its minimal model is no longer smooth. Therefore, the canonical volume $\vol(X):=\vol(K_X)$ can be a small positive rational number. This leads to the following two questions.

\begin{ques}
Is there a Noether-type inequality for varieties of general type of dimension $m$? More precisely, fix dimension $m$, does there exists $a_m, b_m>0$ such that $\vol(X) \ge a_m p_g(X)-b_m$?
\end{ques}

This question is known only up to dimension $2$, and partially known in dimension $3$.
One might want to ask the following weaker question, which is known up to dimension $3$.

\begin{ques}
Fix dimension $m$. Does there exists $a'_m, b'_m>0$ such that $\vol(X) \ge a'_m p_g(X)-b'_m$ for $m$-dimensional varieties of general type with Gorenstein minimal model?
\end{ques}

\begin{ques} \label{lim}
Suppose that the answer to above questions are positive for all $m$. What are $\displaystyle{\lim_{m \to \infty} a_m}$ and $\displaystyle{\lim_{m \to \infty} a'_m}$?
\end{ques}

\begin{ex}
Let $X=X_{28} \subset \bP(1,3,4,5,14)$ be a general weighted hyperplane of degree $28$. If is known that $\vol(X)=\frac{1}{30}$ and $p_g(X)=1$. Let $C$ be a curve of genus $\ge 2$ and $Z=X \times C$. Then $p_g(Z)=g(C)$ and $\vol(Z)=\frac{8}{30} (p_g(Z)-1)$.

In the same spirit, let $X$ be an $m$-dimensional variety $X$ with $p_g(X)=1$ and small volume. Let $Z=X \times C$. Since $$\vol(Z)=(m+1) \vol(X) (2g(C)-2)$$ and $p_g(Z)=g(C)$, it is expected that there is an $m$-dimensional variety $X$ with $p_g(X)=1$ and $\vol(X)$ is small enough such that $\displaystyle{\lim_{m \to \infty} (m+1){\vol(X)}=0}$. Hence it is expected that $\displaystyle{\lim_{m \to \infty} a_m =0}$.
\end{ex}

In this article, we produce series of examples of varieties of general type with small slope $\frac{{\rm Vol}(X)}{p_g(X)}$.
Suppose that $p_g(X)\neq 0$, recall that $d_1$ is the dimension of the image of canonical map $\varphi_{|K_X|}(X)$.
\theoremstyle{theorem}
\newtheorem*{main1}{Example A}
\begin{main1} For any $n\geq3$, there exists an $n$-dimensional minimal smooth projective variety $X$ of general type and $d_1=n-1$ such that
 $${\rm Vol}(X)=\frac{n+1}{n} p_g(X) - \frac{n^2+1}{n}.$$
\end{main1}

If $d_1\leq\dim X-2$, then we find a series of examples of $n$-dimensional varieties of general type, where the slope $\frac{{\rm Vol}(X)}{p_g(X)}$ goes to zero as $n$ increases (cf. Proposition \ref{B}).
\theoremstyle{theorem}
\newtheorem*{main2}{Example B}
\begin{main2}  For each $n\geq2$ and $k\geq1$, there exists an $(n+k+1)$-dimensional smooth projective variety $X$ of general type such that
        $d_1=n$ and 			
        $${\rm Vol}(X)=\frac{n+k+2}{2^k(n+1)}p_g(X)-\frac{(n^2+2n+2)+(n+2)k}{2^k(n+1)}.$$
\end{main2}

The following corollary answers Question \ref{lim} by fixing $n$ but letting $k$ increased.
\begin{cor}  Keep the notation as in Question \ref{lim}. Then $\displaystyle{\lim_{m \to \infty} a_m =0}$.
\end{cor}

Note that we do not claim the constructed varieties in either Example A or B have Gorenstein minimal models. Therefore, we don't know $\displaystyle{\lim_{m \to \infty} a'_m}$  yet.

Our construction can be summarized in the following diagram:
$$\begin{array}{ccccccc}
&&& X_A& &&X_B \\
&&& \downarrow{\tau_A}& &&\downarrow{\tau_B}\\
Y_1\leftarrow Y_2 & \cdots &\leftarrow &Y_n \leftarrow W \leftarrow W_1 &\cdots &\leftarrow& W_k
\end{array}
$$
where $Y_1=\bP^1$, $Y_2=\bF_e$, all the horizontal maps are $\bP^1$-bundles, and vertical maps are double covers (cf. Section 4 and 5).


This paper is organized as the following: In Section 2, we recall a proposition of Kobayashi, indicating the relation of the Noether inequality and varieties of minimal degree, for which our construction is based on. In Section 3, we take towers of $\bP^1$-bundles over a variety of minimal degree and establish some basic properties. In Section 4 and 5, we construct Example A and B respectively. In Section 6, some related open questions are discussed.

\ni\textbf{Acknowledgement.} We would like to thank Meng Chen for the helpful discussion. This work is done partially during the second author's visit at Research Institute for Mathematical Sciences and he thanks for the warm hospitality of the institute. The second author is also supported by National Center of Theoretical Sciences in Taiwan and funded by MOST-104-2115-M-002-011-MY2.

\section{polarized varieties of minimal degree}

Given a non-degenerate projective variety  $X^n \subset \bP^r$ of dimension $n$. It is well-known that $\deg(X) \ge \codim(X, \bP^r)+1$ (cf. \cite{EH}). If the equality holds, then $X$ is called a {\it variety of minimal degree}. Suppose that  $\cL$ is a free and big line bundle on $X$. Then $\varphi_\cL$ defines a generically finite morphism.
Let $Z=\varphi_{\cL}(X)$ be its image in projective space. Then $Z$ is a variety of degree $\cL^n$. Recall that the $\Delta$-genus of a polarized variety $(X, \cL)$ introduced by  T. Fujita is defined as $$\Delta(X, \cL) =\cL^n+n-h^0(X, \cL).$$ It is clear that the image $Z$ is a variety of minimal degree if and only if $(X, \cL)$ is a polarized variety with $\Delta$-genus $\Delta(X, \cL)=0$. Moreover, polarized varieties with $\Delta$-genus zero are classified in \cite{EH, F1,F2,F3}.

The following result of Kobayashi reveals the connection between variety of minimal degree and varieties of general type with small slopes.

\begin{prop}
Let $X$ be an $n$-dimensional minimal variety of general type with at most canonical singularities such that $d_1=n$. Then $\vol(X) \ge 2 (p_g(X) -n)$. Equality holds only if $X$ is Gorenstein and admits a $2:1$ covering over a variety of minimal degree.
\end{prop}

Two dimensional projective varieties of minimal degree are cone over rational normal curves, the embedding of Hirzebruch surface $\bF_e=\bP_{\bP^1}(\cO_{\bP_1}\oplus\cO_{\bP^1}(e))$ by a suitably chosen linear system, or the Veronese surface $\bP^2\subseteq\bP^5$ (cf. \cite[Exercises IV.18.(4)]{Bea}).


\section{Towers of $\bP^1$-bundles}
The purpose of this section is to construct towers of $\bP^1$-bundles and study some basic properties.

Given a lower triangular matrix $(c_{i,j})_{1\le i, j \le n-1}$ with integer entries, we can construct a tower of $\mathbb{P}^1$-bundles successively as follows. Start with $Y_1=\mathbb{P}^1$, $\Sigma_1=\cO_{Y_1}(1)$ and $\cL_1=\cO_{Y_1}(c_{11})$, then we construct $Y_2=\bP_{Y_1} ( \cO_{Y_1} \oplus \cL_1^{-1})$ as a $\bP^1$-bundle over $Y_1$. Let $\Sigma_2 \subset Y_2$ be the distinguished section and fix $\cL_2=\cO_{Y_2}(c_{21} \pi_{21}^*\Sigma_1+ c_{22} \Sigma_2)$.  We then construct $Y_3=\bP_{Y_2} ( \cO_{Y_2} \oplus \cL_2^{-1})$ as a $\bP^1$-bundle over $Y_2$. Inductively, we can construct all the way to get $Y_n$ with distinguished section $\Sigma_n$.

For $k\geq2$, the distinguished section $\Sigma_k$ on $Y_k$ is defined by the quotient $\cL^{-1}_{k-1}$ of $ \cO_{Y_{k-1}} \oplus \cL_{k-1}^{-1}$ on $Y_{k-1}$, which satisfies $\Sigma_k\sim \cO_{Y_k}(1)$ on $Y_k$ with normal bundle $\cN_{\Sigma_k/Y_k}=\Sigma_k|_{\Sigma_k}\cong\cL_{k-1}^{-1}$.

Let $\pi_{i,j}$ denote the induced map from $Y_i$ to $Y_j$ and $\pi_i$ denotes $\pi_{n,i}$ for simplicity. Also, for brevity, we abuse the notation of $\Sigma_i$ with $\pi_{n,i}^*\Sigma_i$ and all other divisors or line bundles as well. It is easy to see that $\textrm{Pic}(Y_n) \cong \bZ^n$ with $\{\Sigma_1, \Sigma_2,...,\Sigma_n\}$ as a set of generators. Hence, we can always write a line bundle on $Y_n$ as $\cM=\cO_{Y_n} (b_1\Sigma_1+b_2\Sigma_2+\cdots b_n \Sigma_n)$ and simply denote it as $\cO_{Y_n}(b_1,\ldots, b_n)$.
An easy inductive computation shows that
\begin{equation}
  -K_{Y_n}=(2+\sum_{i=1}^{n-1} c_{i,1}, 2+\sum_{i=2}^{n-1} c_{i,2}, \ldots, 2+ c_{n-1,n-1}, 2).
\end{equation}

\begin{ex}\label{byn} We consider $Y_n$ to be the $n$-dimensional tower of $\bP^1$-bundle with the following building data of $(n-1) \times (n-1)$ matrix,
$$(c_{i,j})_{1\le i, j \le n-1}=\left(
\begin{array}{ccccc}
e & 0 & 0 & \ldots & 0\\
e & 1 & 0 & \ldots & 0 \\
& & \vdots & &  \\
e & 1 & \ldots & 1 & 0 \\
e& 1 & \ldots & 1 & 1 \\
\end{array}
\right)$$

Then 
\begin{equation}
-K_{Y_n}=((n-1)e+2, \underbrace{n, n-1,\ldots, 3,2}_{n-1}).
\end{equation}

For $i=1,...,n$, we also consider line bundles $\cL_i=\cO_{Y_i}(e, \underbrace{1, \ldots, 1}_{i-1})$ on $Y_i$ or its pull-back on $Y_n$.
\end{ex}

\begin{prop}\label{yn} Keep the notation as in Example \ref{byn}.
The the line bundle $\cM=b_1\cL_1+\cdots+b_n\cL_n$ satisfies the following properties:
\begin{enumerate}[$(a)$]
	\item $\cL_n^n=e$.
	\item $h^0(Y_n, \cL_n)=e+n$.
	\item If $b_i \ge 0$ for all $i$, then $H^1(Y_n, \cM)=0$.
	\item If $b_i \ge 0$ for all $i$, then $|\cM|$ is free.
	\item $\cL_n$ is free and big.

	\item The polarized variety $(Y_n, \cL_n)$ has $\Delta$-genus $0$. In other words, the image of $\varphi_{\cL_n}$ is a 		variety of minimal degree.
	\item The morphism $\varphi_{\cL_n}$ is a birational morphism contracting $\Sigma_n$.
	\item $\cL_n^{n-1} \cdot \Sigma_1=1$ and $\cL_n^{n-1} \cdot \Sigma_i=0$ for $2 \le i \le n$.
\end{enumerate}
\end{prop}
\begin{proof}
Write $\cL_n=\Sigma_n+\pi_{n-1}^*\cL_{n-1}$, then
$$\begin{array}{ll} \cL_n^n&= (\Sigma_n+\pi_{n-1}^* \cL_{n-1})^n \\
                           &= \sum_{i=0}^n C^n_i (\pi_{n-1}^* \cL_{n-1})^{n-i} \Sigma_n^{i} \\
                           &=     \sum_{i=1}^{n} C^n_i (\pi_{n-1}^* \cL_{n-1})^{n-i} \Sigma_n^{i} \\
                           &=  \sum_{i=1}^{n} C^n_i  \cL_{n-1}^{n-i} (-\cL_{n-1})^{i-1} \\
                           \end{array}.$$
Since $\sum_{i=1}^{n} C^n_i (-1)^i=(1+(-1))^n-1=-1$,
by induction $\cL_n^n=\cL_2^2=e.$ This proves $(a)$.

We write $\cL_n=\Sigma_n+\pi_{n-1}^*\cL_{n-1}$ again and by projection formula $$h^0(Y_n, \cL_n) = h^0(Y_{n-1}, \cL_{n-1} \oplus \cO_{Y_{n-1}}) = h^0(Y_{n-1}, \cL_{n-1})+1.$$
Since $h^0(Y_2, \cL_2)=e+2$, one sees $(b)$ by induction.

We will prove $(c)$ by induction on dimension and on $b_i$. The statement is standard on $Y_1=\bP^1$. Since $Y_i$'s are rational, $H^1(Y_i, \cO_{Y_i})=0$ for all $i$. Suppose now that $\cM_n=\sum_{i=1}^{n} b_i \cL_i$ with $b_i \ge 0$.
 We may and do assume that $b_n >0$, otherwise it is reduced to lower dimensional case.
Since $\cL_n-\Sigma_n=\pi_{n-1}^*\cL_{n-1}$, we can write $\cM_n-\Sigma_n=\cM'$ with
$$\cM'=\sum_{i=1}^{n-2} b_i \pi_i^* \cL_i + (b_{n-1}+1) \pi_{n-1}^*\cL_{n-1} + (b_n-1)\cL_n.$$ The standard restriction sequence to $\Sigma_n$ now reads:
$$ 0 \to \cM' \stackrel{+\Sigma_n}{\lra} \cM_n \to \cM_n|_{\Sigma_n} \to 0. \eqno{(3.3)}$$
On $Y_{n-1}$, we consider $\cM_{n-1}=\sum_{i=1}^{n-1} b_i \pi_i^* \cL_i$.
 Through the isomorphism $\Sigma_n \cong Y_{n-1}$ and $\cL_n|_{\Sigma_n}=0$, one sees that  $\cM_n|_{\Sigma_n} \cong \cM_{n-1}$. It follows that $H^1(Y_n, \cM_n)=0$ by induction on dimension and on $b_n$.

 We now prove  the freeness of $\cM_n$, which is $(d)$. Indeed, it suffices to prove the freeness of $\cL_n$.
 By $(b)$, one sees that $|\cL_{n-1}|+\Sigma_n \subsetneq |\cL_n|$. Pick a general $\Gamma \in |\cL_n|$ which does not contain $\Sigma_n$. We claim that $\Gamma$ is irreducible and $\pi_{n-1}|_\Gamma \colon \Gamma \cong Y_{n-1}$ is an isomorphism.

 To this end, we write $\Gamma=\Gamma_h+\Gamma_v$, where $\Gamma_h$ (resp. $\Gamma_v$) denotes the horizontal (resp. vertical) part with respect to $\pi_{n-1}$. Since $\cL_n \cdot l =1$ for a fiber $l$ of $\pi_{n-1}$, the divisor $\Gamma_h$ is irreducible. By our choice of $\Gamma$, $\Gamma_h \cap \Sigma_n$ is an effective divisor (possibly zero) on $\Sigma_n \cong Y_{n-1}$, hence so is $\Gamma \cap \Sigma_n$.

Note that $\Gamma_v$ must be empty. Otherwise, $\Gamma_v \cap \Sigma_n$ is effective and non-empty. Since $\cL_n|_{\Sigma_n}=0$, for any ample divisor $H$ on $Y_n$, this leads to $$0=H^{n-2} \cdot \Gamma \cdot \Sigma_n \ge H^{n-2} \cdot (\Gamma_v \cap \Sigma_n) >0,$$ a contradiction.
 Similarly, $\Gamma_h \cap \Sigma_n$ is also empty.
 Therefore, $\Gamma=\Gamma_h$ is irreducible and smooth as it is then isomorphic to $Y_{n-1}$ by $\Gamma \cdot l =1$.

Since $\cL_n|_{\Sigma_n}=0$, the section $\Gamma$ is disjoint from $\Sigma_n$ and $\cL_n|_{\Gamma}=(\pi^*\cL_{n-1}+\Sigma_n)|_\Gamma\cong\cL_{n-1}$. We have the following exact sequence:
$$ 0 \to \cO_{Y_n} \stackrel{+\Gamma}{\lra} \cL_n \stackrel{}{\lra} \cL_{n-1} \to 0, \eqno{(3.4)}$$ where the last map is the restriction to $\Gamma\cong Y_{n-1}$. By induction on dimension, we assume that $\cL_{n-1}$ is free on $\Gamma \cong Y_{n-1}$.
 Since $H^1(Y_n, \cO_{Y_n})=0$, it follows that $\cL_n$ is free on $\Gamma$. Let $\Gamma$ moves, then one sees that $\cL_n$ is free on $Y_n$ away from $\Sigma_n$.

 Similarly, we consider the exact sequence by restricting to $\Sigma_n$
 $$ 0 \to \pi_{n-1}^* \cL_{n-1} \stackrel{+\Sigma_n}{\lra} \cL_n \to \cO_{\Sigma_n} \to 0. \eqno{(3.5)}$$
 Since  $H^1(Y_n, \pi_{n-1}^* \cL_{n-1})=0$, it follows that $\cL_n$ is free on $\Sigma_n$. This completes the proof of $(d)$.

The statement of $(e)$ and $(f)$ follow from $(a)$,$(b)$ and $(d)$ immediately.
By $(3.4)$ and induction, $\varphi_{\cL_{n}}|_\Gamma \cong \varphi_{\cL_{n-1}}$ defines a birational morphism on $\Gamma$, which extends to a birational morphism $\varphi_{\cL_n}$ on $Y_{n}$. As we saw in the proof of $(d)$,  $\varphi_{\cL_n}$ restricts to a constant map on $\Sigma_n$. This proves $(g)$.

To compute that $\cL_n^{n-1} \cdot \Sigma_i$, let $\Gamma \in |\cL_n|$ as before. Since  $\cL_n|_\Gamma = \cL_{n-1}$,  inductively we get $$ \cL_n^{n-1} \cdot_{Y_n} \Sigma_1 = \cL_{n-1}^{n-2} \cdot_{Y_{n-1}} \Sigma_1=\cdots=\cL_2 \cdot_{Y_2} \Sigma_1=1, $$ and $$ \cL_n^{n-1} \cdot_{Y_n} \Sigma_i = \cL_{n-1}^{n-2} \cdot_{Y_{n-1}} \Sigma_i=\cdots =\cL_i^{i-1} \cdot_{Y_i} \Sigma_i=0,$$
as $\cL_i|_{\Sigma_i}=0$ for $i\geq2$.
\end{proof}

\section{Construction of Example A}

\begin{ex} \label{w}
We keep the notation as in Section 3. We take $W=\bP_{Y_n}(\cO_{Y_n} \oplus \cL_n^{-2})$, which is $(n+1)$-dimensional. In other words, $W$ is constructed by using the following building data of $n \times n$ matrix
$$\left(
\begin{array}{ccccc}
e & 0 & 0 & \ldots & 0\\
e & 1 & 0 & \ldots & 0 \\
& & \vdots & &  \\
e & 1 & \ldots & 1 & 0 \\
2e& 2 & \ldots & 2 & 2 \\
\end{array}
\right)$$

Let $\Sigma_{n+1} \subset W$ be the distinguished section: $\Sigma_{n+1}\sim\cO_W(1)$ and $\cN_{\Sigma_{n+1}}=\Sigma_{n+1}|_{\Sigma_{n+1}}\cong\cL_n^{-2}$.
By (3.1), one has
$$-K_W= ((n+1)e+2,\underbrace{(n+2),(n+1),\cdots,5,4}_{n-1},2).$$
Also, we define
$$ \cL_{n+1}=(2e, \underbrace{2, 2, \ldots,2}_{n-1},1)=2\pi_{n}^*\cL_n+\Sigma_{n+1}. \eqno{(4.1)}$$
\end{ex}

In what follows, we usually identify $\Sigma_{n+1}$ with $Y_n$ and identify $\pi_n^* \cF|_{\Sigma_{n+1}}$ with $\cF$ on $Y_n$ for any sheaf $\cF$on $Y_n$.

\begin{prop} \label{wn+1}
Let $W$ be the variety as in Example \ref{w} and write a line bundle as $\cM=\sum_{i=1}^{n+1} b_i \cL_i$. We have the following similar properties as in Proposition \ref{yn}.
    \begin{enumerate}[$(a)$]
	\item $h^0(W, \cL_{n+1}) \ge 2$, $\Sigma_{n+1}$ is not a base divisor of $|\cL_{n+1}|$, and a general member $\Gamma \in |\cL_{n+1}|$ is smooth and isomorphic to $Y_n$.
        \item If $b_i \ge 0$ for all $i$, then $H^1(W_{n+1}, \cM)=0$.
        \item If $b_i \ge 0$ for all $i$, then $|\cM|$ is free.
    \end{enumerate}
\end{prop}

\begin{proof}
The statement $(a)$ follows from the standard sequence by restricting to $\Sigma_{n+1}$ as in the proof of Proposition \ref{yn}. The proof of $(b)$ and $(c)$ are almost the same as those in Proposition \ref{yn}. For example, let $\cM_{n+1} = \sum_{i=1}^{n+1} b_i \cL_i$. Since $\cL_{n+1}-\Sigma_{n+1}=2\cL_n$, one gets $$\cM'_{n+1}=\cM_{n+1}-\Sigma_{n+1}=\sum_{i=1}^{n-1} b_i \cL_i+ (b_n+2)\cL_n + (b_{n+1}-1) \cL_{n+1}.$$  The restriction sequence to $\Sigma_{n+1}$ reads
$$ 0 \to \cM'_{n+1} \stackrel{+\Sigma_{n+1}}{\lra} \cM_{n+1} \to \cM_{n} \to 0,$$ where $\cM_n=\sum_{i=1}^n b_i \cL_i$ on $\Sigma_{n+1} \cong Y_n$. One can then proceed by induction similarly.

We leave the details to the readers.
\end{proof}

To construct our Example A, we need to consider the following line bundles:
 $$ \begin{array}{l}
  \cT_{n+1}=((2n+6)e, 2n+6, 2n+4, \ldots, 12, 10, 5)\\
  \mathcal{B}_{n+1} =((n+3)e, n+3, n+2, \ldots, 6, 5, 3)\equiv\frac{1}{2}(\cT_{n+1}+\Sigma_{n+1}) \\
  \cN_{n+1}=(2e-2, 1, 1, \ldots, 1, 1)=K_W+\mathcal{B}_{n+1} \\
  \end{array}
  $$
and the $\bQ$-line bundle
 $$ \begin{array}{l}
 \cH_{n+1}=(2e-2, 1, 1, \ldots, 1, 1/2)
 \end{array}.
 $$
 The choices of these line bundles are inspired by the work of Kobayashi.


\begin{lem}\label{br} The linear system $|\cT_{n+1}|$ is free and we can choose a general member $T_{n+1} \in |\cT_{n+1}|$ such that $T_{n+1}+ \Sigma_{n+1} \in |2\cB_{n+1}|$ is a simple normal crossing divisor. 
\end{lem}
\begin{proof}
Since $\cT_{n+1}=5 \cL_{n+1}+2\sum_{i=1}^{n-1} \cL_i$, it is free by Proposition \ref{wn+1}. Moreover, $\cT_{n+1}|_{\Sigma_{n+1}}=2\sum_{i=1}^{n-1} \cL_i$ is free and the restriction map is surjective thanks to the vanishing of $$H^1(W, \cT_{n+1}-\Sigma_{n+1})=H^1(W, 4\cL_{n+1}+2\cL_n+\sum_{i=1}^{n-1} \cL_i)$$
as proved in Proposition \ref{wn+1}. The claim follows by Bertini's Theorem.
\end{proof}


\begin{ex}[=Example A] We assume that $e \ge 2$. 
Consider the double cover $\tau: X\to W$ branching along a general simple normal crossing divisor $T_{n+1}+\Sigma_{n+1} \in |2\cB_{n+1}|$ (cf. Lemma \ref{br}). It is clear that $X$ has singularities of type $A_1 \times \bA^{n-1}$, which is Gorenstein and canonical. Note that
$$ \begin{array}{l}
K_{X}=\tau^*(K_W+\cB_{n+1})=\tau^*(2e-2,1\ldots,1)=\tau^*\cN_{n+1},
\end{array}
$$
is nef as $\cN_{n+1}=(e-2)\Sigma_1+\cL_{n+1}$ is nef  by Proposition 4.2. Blow up the singular locus of $X$, which is crepant, we obtain a smooth minimal variety $\tilde{X}$ with
$h^0(\tilde{X}, mK_{\tilde{X}})=h^0(X, mK_X)$ for all $m$.
\end{ex}



Before moving on, we fix some notations.
Let $\Sigma_0=\tau^{-1}(\Sigma_{n+1})$ and $T_0=\tau^{-1}(T_{n+1})$ be the ramification divisors on $X$. Then
$\tau^*(\Sigma_{n+1})=2\Sigma_0,\ \tau^*T_{n+1}=2T_0,$ and
$$\Sigma_0|_{\Sigma_0}=\frac{1}{2}\tau^*\Sigma_{n+1}|_{\Sigma_0}=-\cL_{n}.$$

\begin{prop}\label{A}
Keep the above notations. Then the following holds for the variety $X_A:=\tilde{X}$ as constructed in Example 4.4.
\begin{enumerate}[$(a)$]
	\item The geometric genus $$p_g(X_A)=h^0(X_A, K_{X_A}) = h^0(W, \cN_{n+1}) = (n+1)e-n.$$
	\item $H:=K_{X_A}-\Sigma_0=\tau^*\pi_{n}^*\cN_n+\Sigma_0\equiv\tau^*\cH_{n+1}$ is  big and semiample.
	\item The volume $\vol(X_A) = 2\cH_{n+1}^{n+1}=(n+2)e-2(n+1)$.
	\item 
		One can identify the canonical image of $X_A$ with $\varphi_{\cN_n}(Y_n)$, which is $n$-dimensional.
\end{enumerate}
\end{prop}
\begin{proof} Since the resolution $\tilde{X}\rightarrow X$ is crepant, it is enough to work on $X$.
Let $\cN_i=(2e-2,1, \ldots, 1)$ on $Y_i$ for $i=1,..., n$, then one has $\cN_i=(e-2)\Sigma_1+\cL_i$. In particular, $\cN_i$ is nef. By projection formula, one obtains
\begin{align*}
h^0(Y_i, \cN_i)&= h^0(Y_{i-1}, \cN_{i-1})+ h^0(Y_{i-1}, \pi_{i,1}^*(e-2)\Sigma_1) \\
&= h^0(Y_{i-1}, \cN_{i-1})+e-1.
\end{align*}
Since $h^0(\cN_1)=2e-1$, we have $h^0(\cN_n) = (n+1)e -n$ by induction. Since $p_g(W)=0$ as $W$ is rational, $\tau_*\cO_X=\cO_W\oplus\cO_W(-\cB_{n+1})$, and $(\pi_{n+1,n})_*\cO_W(1)=\cO_{Y_n}\oplus\cL_n^{-2}$, it follows again from projection formula that
\begin{align*}
h^0(K_X)=h^0(W, \cN_{n+1})=h^0(Y_n, \cN_n).
\end{align*}
This proves $(a)$.

We now prove $(b)$. Since $2 \cH_{n+1}=  \cL_{n+1}+(2e-4) \Sigma_1$, it follows that $\cH_{n+1}$ is semiample by Proposition \ref{wn+1} $(c).$

It remains to show $\cH_{n+1}^{n+1}>0$. Since $|\cL_{n+1}|$ is base point free, a general member $\Gamma'\in|\cL_{n+1}|$ is smooth and $\Gamma'\cong Y_n$ as $\Gamma'.l=1$ for $l$ a fiber of $W\rightarrow Y_n$. Since $\cL_{n+1}|_{\Sigma_{n+1}}\cong\cO_{Y_n}$, it follows that  $\Gamma'$ is disjoint from $\Sigma_{n+1}$. Therefore  $$\cL_{n+1}|_{\Gamma'}=(2\cL_n+\Sigma_{n+1})|_{\Gamma'}\cong2\cL_n. \eqno{(4.2)}$$
Now we can easily compute
\begin{align*} (2\cH_{n+1})^{n+1}=&((2e-4)\Sigma_1+\cL_{n+1})^{n+1}\\
						  =&\cL_{n+1}^{n+1}+(n+1)(2e-4)\cL_{n+1}^n\cdot_{W}\Sigma_1\\
						  =&(2\cL_n)^n+(n+1)(2e-4)(2\cL_n)^{n-1}\cdot_{Y_n}\Sigma_1 \\
						  =&2^ne+(n+1)(2e-4)2^{n-1},
\end{align*}
where we have used $(a)$ and $(h)$ in Proposition \ref{yn}. Hence
$$2(\cH_{n+1}^{n+1})=(n+2)e-2(n+1)>0.$$
This proves assertion $(b)$.

We now prove $(c)$. By definition, if there is a constant $c$ such that
$$h^0(X, mK_X)= c\cdot\frac{m^{n+1}}{(n+1)!}+\text{l.o.t.}\ {\rm for\ all}\ m\gg0,$$
then $\vol(X)=c$. We will show that $c=2 \cH_{n+1}^{n+1}$.
By projection formula, we have $$h^0(X, mK_X)=h^0(W, m\cN_{{n+1}})+h^0(W, m\cN_{n+1}-\cB_{n+1}).$$
The proof can be completed by the following three steps.

\noindent
{\bf Claim 1.} If $m$ is even, then
$$h^0(W, m\cN_{{n+1}})=h^0(W, m\cH_{{n+1}})=(\cH_{n+1}^{n+1})\cdot\frac{m^{n+1}}{(n+1)!}+\text{l.o.t.}$$

For any $j \ge 1$, we consider the exact sequence
$$ 0 \to m \cH_{n+1} +(j-1)\Sigma_{n+1} \to m\cH_{n+1}+ j \Sigma_{n+1} \to (m\cH_{n+1}+ j \Sigma_{n+1})|_{\Sigma_{n+1}} \to 0.$$

Since $2 \cH_{n+1} = \cL_{n+1}+(2e-4)\Sigma_1$, $\cL_{n+1}|_{\Sigma_{n+1}}=\cO_{Y_n}$, and $\Sigma_{n+1}|_{\Sigma_{n+1}}=-2\cL_n$, one sees that
\begin{align*}(m\cH_{n+1}+ j \Sigma_{n+1})|_{\Sigma_{n+1}} =&\ \frac{m}{2}(2e-4) \Sigma_1 -2j \cL_n
\end{align*}
which can not be effective. Therefore,
$$H^0(W, m \cH_{n+1}) \stackrel{+\frac{m}{2} \Sigma_{n+1}}{\lra} H^0(W, m \cH_{n+1}+\frac{m}{2} \Sigma_{n+1})=H^0(W, m \cN_{n+1})$$ is an isomorphism. By the asymptotic Riemann-Roch formula for $\cH_{n+1}$, the Claim follows.

\noindent
{\bf Claim 2.} $\vol(\cN_{n+1})=\vol(\cH_{n+1})=\cH_{n+1}^{n+1}$.

From the definition of volume, we get
\begin{align*} \vol(2\cN_{n+1})=\limsup_{m\rightarrow\infty}\frac{h^0(2m\cN_{n+1})}{m^d/d!}=\limsup_{m\rightarrow\infty}\frac{h^0(2m\cH_{n+1})}{m^d/d!}=\vol(2\cH_{n+1}),
\end{align*}
where $d=\dim X=n+1$. Since $\cN_{n+1}$ is big, $\cH_{n+1}$ is free and big, and volume is homogeneous for big divisors by \cite[Proposition 2.2.35]{P1}, we get $\vol(\cN_{n+1})=\vol(\cH_{n+1})=\cH_{n+1}^{n+1}$.

\noindent
{\bf Claim 3.} End of proof of $(c)$.

Since it is easy to find an integer $d$ such that $d \cN_{n+1}-\cB_{n+1}$ is effective, we obtain that
$$ h^0(W, (m-d) \cN_{n+1}) \le h^0(W, m\cN_{n+1}-\cB) \le h^0(W, m \cN_{n+1}),$$
and hence for $m\gg0$,
$$h^0(W, m\cN_{n+1}-\cB_{n+1})=(\cH_{n+1}^{n+1})\cdot\frac{m^{n+1}}{(n+1)!}+\text{l.o.t.}$$
Combined with Claim 1 and 2, we finish the proof of the statement $(c)$.

For $(d)$, it is easy to see that by the projection formula, both the maps
$$H^0(Y_{n}, \cN_{n}) \stackrel{\pi_{n}^*}{\lra} H^0(W, \cN_{n+1}) \stackrel{\tau^*}{\lra} H^0(X, K_X)$$
are isomorphisms. Therefore, $|K_X|\cong\tau^* \pi_{n}^* | \cN_n|$. It follows that the map defines by $|K_X|$ is the composition of the map induced by $ |\cN_n|$ and $\pi_{n+1,n} \circ \tau \colon X \to Y_n$.
Since $|\cL_n| \hookrightarrow |\cN_n|$ as long as $e \ge 2$, where $|\cL_n|$ is free and induces a birational morphism by Proposition \ref{yn}.$(g)$, this shows that the image of $\varphi_{|K_X|}$ is $n$-dimensional and hence completes the proof.
\end{proof}

\begin{rem} One can actually prove that $|H|$ is base point free and contracts $\Sigma_0$ to $\bP^1$, which can be identified with $\pi:Y_n\rightarrow Y_1=\bP^1$ as  $H|_{\Sigma_0}=\cN_n-\cL_n=\pi^*(e-2)\Sigma_1$.
\end{rem}


\section{Construction of Example B}
The purpose of this section is to establish Example B. We start by considering a higher tower of $\bP^1$-bundles.

\begin{ex}
Fix $k\geq1$. For any $ 1\le l\le k$, we construct $\bP^1$-bundles inductively as follows: Start with $W_0=W$ and $\cL_{n+1}$ as in (4.1). Take $W_{l}:=\bP_{W_{l-1}}(\cO \oplus \cL_{n+l}^{-1})$ and consider the line bundle $\cL_{n+1+l}:=\Sigma_{n+1+l} + \pi^*\cL_{n+l}$ on $W_l$, where $\Sigma_{n+1+l}\sim\cO_{W_l}(1)$ is the distinguished section with $\cN_{\Sigma_{n+1+l}/W_l}=\Sigma_{n+1+l}|_{\Sigma_{n+1+l}}\cong\cL_{n+l}^{-1}$. By abusing the notation, we still write $\pi:W_l\rightarrow W_{l-1}$. We end up with a $(n+1+k)$-dimensional tower of $\bP^1$-bundles by considering the following building data of $(n+k) \times (n+k)$ matrix,
$$\left(
\begin{array}{ccccccccc}
e & 0 & 0 &  & & \ldots &  & & 0\\
e & 1 & 0 & & &\ldots   &  & & 0 \\
& & & &  \vdots  & & & & \\
e & 1 & \ldots & 1 &0 &   &  &\ldots & 0\\
2e& 2 & \ldots & 2 &2 & 0&   & \ldots& 0 \\
2e& 2 & \ldots & 2 &2 & 1 & 0 & \ldots & 0 \\
& & & & \vdots  & & & &\\
2e& 2 & \ldots & 2 &2 & 1 & 1 &\ldots & 1 \\
\end{array}
\right).$$
This matrix is the expansion of the $n\times n$ matrix in Section 4 by the last $k$ rows.
It is clear that we have $\cL_{n+1+l} \cdot \Sigma_{n+1+l}\equiv0$ as $(n+l-1)$-cycle on $W_{l}$ for $l \ge 1$.

Note that the $(n+1+k)$-dimensional tower of $\bP^1$-bundle $W_k$ has canonical divisor of the following form  $$-K_{W_{k}}=\displaystyle{((2k+n+1)e+2,\underbrace{2k+n+2,\dots,2k+4}_{n-1} , \underbrace{k+2, \dots,3,2}_{k+1})}.$$
\end{ex}

We first establish some  properties generalizing those in Section 4.
\begin{lem} In the setup of the above notations, we have for all $l\geq0$ the following properties:
\begin{enumerate}[$(a)$]
\item $\cL_{n+1+l}$ is free and big;
\item $\cL_{n+1+l}^{n+1+l}=2^n e$;
\item $\cL_{n+1+l}^{n+l}\cdot \Sigma_1=2^{n-1}$;
\item $(\cL_{n+1+l})^{n+l} \cdot_{W_{l}} \Sigma_{n+1+l}=0 $. 
\end{enumerate}
\end{lem}

\begin{proof} All the statements hold for $l=0$ from Proposition \ref{yn} and the proof of Proposition \ref{wn+1}.
We assume now $l\geq1$.

Use $\cL_{n+1+l}|_{\Sigma_{n+1+l}}=0$ and $H^1(W_{n+1},\cL_{n+1})=0$, one can show inductively that $H^1(W_{l},\cL_{n+1+l})=0$. It follows that $H^1(W_{n+1+l},\cL_{n+1+l}-\Sigma_{n+1+l})=0$ and we now apply the argument as in Proposition \ref{wn+1}. This proves freeness of $\cL_{n+1+l}$ and $(a)$ follows once we prove $(b)$.

Since $\cL_{n+1+l} \cdot \Sigma_{n+1+l}\equiv0$ for $l\geq1$, equality in $(b)$ follows from
\begin{align*}
\cL_{n+1+l}^{n+1+l} =&\ \cL_{n+1+l}^{n+l}\cdot (\Sigma_{n+1+l}+\pi^*\cL_{n+l})
			       =\cL_{n+1+l}^{n+l} \cdot \pi^* \cL_{n+l}\\
				&\ \vdots\\
			       =&\ (\Sigma_{n+1+l}+\pi^*\cL_{n+l}) \cdot \pi^* (\cL_{n+l})^{n+l}
		                = \Sigma_{n+1+l} \cdot \pi^* (\cL_{n+l})^{n+l}\\
			       =&\ \cL_{n+l}^{n+l}=\cdots=\cL_{n+1}^{n+1}=2^ne.
\end{align*}

For $(c)$, the proof is the same as Proposition \ref{yn}.$(h)$: there is a section $\Gamma_l\cong W_l$ disjoint from $\Sigma_{n+1+l}$, then
by restriction to $\Gamma_l,$ $$\cL_{n+1+l}^{n+l}\cdot \Sigma_1=\ldots=\dots\cL_{n+1}^n\cdot_W\Sigma_1=2^{n-1}.$$

The last statement is proved in the same way as $(c)$.
\end{proof}
The next lemma is useful.
\begin{lem} \label{neg}
	Denote by $$Z_j=\begin{cases} Y_j, &\ 1\leq j\leq n\\ W_{j-n-1}, &\ n+1\le j\leq n+k+1\end{cases}.$$
	For any $1\leq  j \le n+k+1 $, $h^0(Z_j, \cO(a_1, a_2, \ldots, a_j))=0$ if $a_i<0$ for some $1\leq i\leq j$.
\end{lem}
\begin{proof} It is easy to see that if $a_j<0$, then $h^0(Z_j,\cO(a_1,\ldots,a_j))=0$. Otherwise, suppose that $D$ is an effective divisor of $\cO(a_1,\ldots,a_j)$. By intersecting with a fibre of $Z_j\rightarrow Z_{j-1}$, one sees a contradiction immediately.

If now $a_j<0$ but $a_{j+1}\geq0$, then
$$h^0(Z_{j+1},\cO(a_1,\ldots,a_{j+1}))=h^0(Z_j,\cO(a_1,\ldots,a_j)\otimes{\rm Sym}^{a_{j+1}}(\cO\oplus G_j^{-1}))$$
where $G_j$ is either $(e,1,\ldots,1)$ or (a multiple of) $(2e,2,\ldots,2,1,\ldots,1)$. This forces the last summand on $Z_j$ to be negative and hence $$h^0(Z_{j+1},\cO(a_1,\ldots,a_{j+1}))=0\
{\rm if}\ a_j<0\ {\rm and}\ a_{j+1}\geq0.$$ This implies $h^0(Z_{j+1},\cO(a_1,\ldots,a_{j+1}))=0$ if $a_j<0$ with arbitrary  $a_{j+1}\in\bZ$. The lemma now follows by induction via projection formula.
\end{proof}

On $W_k$, we take the line bundles:
$$ \begin{array}{ll}
\cL_{n+1+k} &=(2e, \underbrace{2, 2, \ldots,2}_{n-1},\underbrace{1,\ldots,1}_{k+1})\\
\cN_{n+1+k} &=(2e-2, \underbrace{1, \ldots,1}_{n+k})=\pi_k^*\cN_{n+1}+\sum_{l=1}^{k} \Sigma_{n+1+l}\\
\mathcal{B}_{n+1+k} &=((n+2k+3)e, \underbrace{2k+n+3,\ldots, 2k+5}_{n-1},k+3, \underbrace{k+2, \cdots, 4,3}_{k}) \\
\cT_{n+1+k} &=((2n+4k+6)e, \underbrace{ 2n+4k+6,\ldots, 4k+10}_{n-1}, 2k+5, \underbrace{2k+4, \cdots, 8}_{k-1}, 5)\\
\end{array}$$


It is easy to check that
$$2\cB_{n+1+k}=\cT_{n+1+k}+\Sigma_{n+1+k} +\Sigma_{n+1}$$ and
\begin{align*}\cT_{n+1+k}=\begin{cases}
\sum_{i=2}^{n-1} 2\cL_i+2\cL_{n+1} +5 \cL_{n+2}, &\ k=1\\
\sum_{i=2}^{n-1} 2\cL_i+\cL_{n+1} + \sum_{i=n+2}^{n+k-1} 2 \cL_i + 3 \cL_{n+k}+5 \cL_{n+1+k}, & \ k\geq2
\end{cases}
\end{align*}

Therefore, $\cT_{n+1+k}$ is free. Pick a general $F \in |\cT_{n+1+k}|$, it is clear that
$D:= F + \pi^*_k\Sigma_{n+1} + \Sigma_{n+1+k}$ is a simple normal crossing divisor on $W_k$.
\begin{ex}[=Example B]
Now take the double cover $\tau_k:X_k\rightarrow W_{k}$ ramified along the divisor $D$. Locally $X_k$ has singularity type $(u^2=xyz)\times\bA^{n+k-2}$ at triple intersection of $D$ and $(w^2=xy)\times\bA^{n+k-1}$ at double locus of $D$, where the first one is of $cD$-type terminal singularities and the second one is of $A_1$-type canonical singularities.

Notice that  $K_{X_k}=\tau_k^*(K_{W_k}+B_{n+1+k})=\tau_k^*\cN_{n+1+k}$ is not nef:
Let $C$ be a curve on $W_k$ contained in $\Sigma_{n+1+k}\cong W_{k-1}$. Then
\begin{align*}\cN_{n+1+k}.C=&\ (\cN_{n+k}-\cL_{n+k})\cdot_{W_{k-1}}C\\
					=&\ (-2,\underbrace{-1,\ldots,-1}_{n-1},0,\ldots,0)\cdot_{W_{k-1}}.C\\
					=&\ (-2,-1,\dots,-1)\cdot_{Y_n}C_n,
\end{align*}
where $C_n$ is the pushforward of $C$ on $Y_n$. Hence $\cN_{n+1+k}.C<0$ if $C_n$ is supported on the fibres of $Y_{n}\rightarrow Y_{n-1}$.

Now take $X_B\rightarrow X_k$ to be a resolution of singularities. Notice that $H^0(X_B,mK_{X_B})\cong H^0(X_k,mK_{X_k})$ for all $m\geq0$.
\end{ex}

We will compare $\cN_{n+1+k}$ with
\begin{align*}
\cH_{n+1+k}&\colon =\cN_{n+1+k}-\frac{1}{2} \sum_{j=0}^{k} \Sigma_{n+1+j}.
\end{align*}
Observe that $2\cH_{n+1+k}=(2e-4)\Sigma_1+\cL_{n+1+k}$ is big and free by Lemma 5.1.

\begin{prop}\label{B}
Keep the above notation, then the following holds for $X_B$.
\begin{enumerate}[$(a)$]
	\item The geometric genus is give by 
		$$p_g(X_B)=h^0(X_B, K_{X_B}) = h^0(Y_n, \cN_n) = (n+1)e-n.$$
	\item The volume is $$\vol(X_B) = 2\cH_{n+1+k}^{n+1+k}=\frac{n+k+2}{2^k}e-\frac{(n+k+1)}{2^{k-1}}.$$
	\item The image of the canonical map of $X_B$ can be identified with $\varphi_{\cN_n}(Y_n)$, which is $n$-dimensional. 
\end{enumerate}
\end{prop}
\begin{proof} It is enough to work on $X_k$, for which we denote by $X$ in the following computation since there should be no confusion. Similarly, we denote by $\cB$, $\cN$, and $\cH$ the corresponding divisors on $X$.

Statement $(a)$ can be proved via projection formula and Lemma \ref{neg}:
\begin{align*}
H^0(X_k,\cN_{n+1+k})=&H^0(W_k,\cN_{n+1+k})\\
						=&H^0(W_{k-1},\cN_{n+k})\oplus H^0(W_{k-1},\cN_{n+k}-\cL_{n+k})\\
						=&H^0(W_{k-1},\cN_{n+k})=\cdots=H^0(W,\cN_{n+1})\\
						=&H^0(Y_n,\cN_n).	
\end{align*}
Hence part $(c)$ also follows from Proposition \ref{A}.

Similar to the proof of Proposition \ref{A}, we have
$$h^0(X_k,mK_{X_k})=h^0(m\cN_{n+1+k})+h^0(m\cN_{n+1+k}-\cB_{n+1+k})$$

\noindent
{\bf Claim :} $h^0(X_k,m \cN_{n+1+k})= h^0(X_k,m\cH_{n+1+k})$ for $m$ even.
\begin{proof}[Proof of the Claim]
For $l\geq0$ and $m'=m $ or $m/2$ we consider
$$\cN'_{n+1+l}=\cO(\underbrace{(2e-2)m, \ldots, m}_n,\underbrace{m', \ldots, m'}_{l+1}).$$
One has
\begin{align*} &h^0(X_k,\cN'_{n+1+k})\\=\ &h^0(W_k, \cO((2e-2)m, \ldots, m, m', \ldots, m'))\\
=\ &\sum_{i_k=0}^{m'} h^0(W_{k-1}, \cN'_{n+k}-i_k \cL_{n+k}) \\
=\ &\sum_{i_k=0}^{m'} \sum_{i_{k-1}=0}^{m'-i_k} h^0(W_{k-2}, \cN'_{n+k-1}-(i_k+i_{k-1}) \cL_{n+k-1} ) \\
&\ \ \vdots\\
=\ &\sum_{i_k=0}^{m'} \sum_{i_{k-1}=0}^{m'-i_k} \cdots \sum_{i_1=0}^{m'-(i_k+\cdots+i_{2})} h^0(W, \cN'_{n+1}-(\sum_{j=1}^{k}i_j) \cL_{n+1}) \\
=\ &\sum_{i_k=0}^{m'} \sum_{i_{k-1}=0}^{m'-i_k} \cdots \sum_{i_1=0}^{m'-(i_k+\cdots+i_{2})} \sum_{i_0=0}^{m'-(i_k+\cdots+i_1)} h^0(Y_n, m\cN_{n}-2(m'-i_0)\cL_{n}) \\
=\ &\sum_{i_k=0}^{m'} \sum_{i_{k-1}=0}^{m'-i_k} \cdots \sum_{i_1=0}^{m'-(i_k+\cdots+i_{2})} \sum_{\ell_0=i_k+\cdots+i_1}^{m'} h^0(Y_n, m\cN_{n}-2\ell_0\cL_{n})
\end{align*}

 Note that  by Lemma \ref{neg}, the line bundle
\begin{align*} &m\cN_n-2\ell_0\cL_n=\cO(m(2e-2)-2e\ell_0, m-2\ell_0, \ldots,m-2\ell_0)
\end{align*}
has vanishing $h^0$ if any of $i_k >m/2$, $i_k+i_{k-1} > m/2, \ldots, \sum_{j=2}^k i_j >m/2$ holds.
Hence in the following computation, the summation up to $m$ is equal to the summation only up to  $m/2$.
More precisely,
\begin{align*} h^0(X_k,m\cN_{n+1+k})=\ &\sum_{i_k=0}^m \sum_{i_{k-1}=0}^{m-i_k} \cdots \sum_{i_1=0}^{m-(i_k+\cdots+i_{2})} \sum_{\ell_0=i_k+\cdots+i_1}^{m} h^0(Y_n, m\cN_{n}-2\ell_0\cL_{n}) \\
=\ &\sum_{i_k=0}^{m/2} \sum_{i_{k-1}=0}^{m/2-i_k} \cdots \sum_{i_1=0}^{m/2-(i_k+\cdots+i_{2})} \sum_{\ell_0=i_k+\cdots+i_1}^{m/2} h^0(Y_n, m\cN_{n}-2\ell_0\cL_{n}) \\
=\ &h^0(X,m\cH).
\end{align*}
Which verifies the Claim.
\end{proof}

Now by the same reasoning as in Proposition \ref{A}, we get $\vol(X_k)=2\vol(\cN_{n+1+k})=2\vol(\cH_{n+1+k})=2\cH_{n+1+k}^{n+1+k}$.
From $\cH \equiv (e-2)\Sigma_1+\frac{1}{2}\cL_{n+1+k}$ and Lemma 5.1, it follows that
\begin{align*}
2\cH_{n+k+1}^{n+k+1} &= 2\left((\frac{1}{2} \cL_{n+1+k})^{n+1+k} +(e-2)(n+1+k)\frac{1}{2^{n+k}} (\cL_{n+1+k}^{n+k}\cdot \Sigma_1)\right) \\
&=2 \frac{2^n e }{2^{n+k+1}} +(e-2)(n+k+1) \frac{2^{n-1}}{2^{n+k-1}}\\
&=\frac{n+k+2}{2^k}e-\frac{(n+k+1)}{2^{k-1}}.
\end{align*}
This completes the proof.
\end{proof}
From the above proposition, we have
$$\vol(X_B)=\frac{n+k+2}{2^k(n+1)}p_g(X_B)-\frac{(n^2+2n+2)+(n+2)k}{2^k(n+1)}.$$
The variety $X_B$ is the required variety in Example B.

\section{Related Problems}
Our examples show that the geography, or the relation between canonical volume and genus, depends on $n$, the dimension  of variety and on  $d_1$, the dimension of image of the canonical map as well. It is expected that one can say a lot more if $d_1=n-1$. For example, it is reasonable to ask the following several questions.

\begin{ques}  Suppose that $X$ is a $n$-fold minimal model with Gorenstein  singularities. Suppose that $d_1=\dim \varphi_{K_X}(X) = n-1$. Is  it true that  either $K_X^n \ge 2(p_g-(n-1))$ or $X$ is fibered by curves of genus $2$?
\end{ques}
By a similar argument to Kobayashi's proof, one can prove the following result, partially answers the above question.
\begin{prop}
Let $X$ be a smooth minimal $n$-fold of general type. Suppose that $ d_1=\dim \varphi_{K_X}(X) = n-1$. Then $K_X^n \ge p_g-(n-1).$ More precisely, either $K_X^n \ge 2(p_g-(n-1))$ or $X$ is fibered by curves of genus $\le 1+\lfloor \frac{n}{2} \rfloor$.
\end{prop}
\begin{proof}
We may write $|K_X|=|D|+B$, where $D$ is the mobile part and $B$ is the base divisor.
Let $\pi \colon X'\to X$ be a resolution of $|D|$ so that the proper transform $|D'|$ is free.
That is, $\pi^*|D|=|D'|+B'$ and $\pi^*|K_X|=|D'|+F,$ where $F=\pi^*B+B'$. Let $C$ be a general fibre curve of the morphism $\varphi_{|D'|}$ induced on $X'$.

Since $\pi$ consists of a sequence of blowup along smooth centers, we have $(n-1)B' \ge K_{X'/X}$ and hence $(n-1)F \ge K_{X'/X}$, where $K_{X'/X}=K_{X'}-\pi^* K_X$.

Notice that $\varphi_{|D'|}$ has image of dimension $n-1$ and hence $D'^{n-1} \ge ((h^0(D')-n+1) C$. Therefore, we have the following
$$\begin{array}{ll}
K_X^n=\pi^*K_X^n & \ge \pi^*K_X^{n-1} \cdot (D'+F) \\
&\ \ \ \vdots \\
                 & \ge \pi^*K_X \cdot D'^{n-1}\\
                 & \ge ((h^0(D')-n+1) \pi^*K_X \cdot C \\
                 &= (p_g(X)-n+1)\pi^*K_X \cdot C.

\end{array} \eqno{(6.1)} $$

Since $K_C=K_{X'}|C$, it follows that $K_{X'}\cdot C = 2g(C)-2 >0$.
Now $$ n \pi^*K_X \cdot C = (n-1) F \cdot C + \pi^*K_X \cdot C \ge K_{X'/X} \cdot C + \pi^*K_X \cdot C = K_{X'} \cdot  C >0. \eqno{(6.2)}$$

Hence $ \pi^*K_X \cdot C  \ge 1$ and therefore $K_X^n \ge p_g(X)-n+1$ by (6.1).

Suppose that $g(C) > 1+\lfloor \frac{n}{2} \rfloor$. Then clearly, $g(C) > 1+\frac{n}{2}$ and hence $2g(C)-2 >n$.
By (6.2), $\pi^*K_X \cdot C >1$ and hence $\pi^*K_X \cdot C \ge 2$. Together with (6.1) , one sees that either $K_X^n \ge 2(p_g-(n-1))$ or $X$ is fibered by curves of genus $\le 1+\lfloor \frac{n}{2} \rfloor$.
\end{proof}

\noindent
\begin{ques}
Suppose that  $X$ is a $n$-fold minimal model of general type with at worst Gorenstein singularities. Suppose furthermore that $d_1=\dim \varphi_{K_X} = n-1$. Does the relation  $$K_X^n \ge \frac{n+1}{n} p_g - \frac{n^2+1}{n}$$ holds?
\end{ques}

\noindent
\begin{ques}
Suppose that  $X$ is a $n$-fold minimal model of general type with at worst Gorenstein  singularities. Suppose  furthermore that $d_1=\dim \varphi_{K_X} = n-1$ and $$K_X^n = \frac{n+1}{n} p_g - \frac{n^2+1}{n}.$$
Is it true that $X$ is a double covering over a variety of minimal degree, similar to  example $X_A$?
\end{ques}

It is known in dimension three that the above equality characterizes the variety $X$ essentially as Kobayashi's example if $p_g\geq7$(cf. \cite{CH}). The same proof partially works in the setting of Question 6.4 with the extra assumption $n-1=2g(C)-2$.



\begin{thebibliography}{ELMNP}
	\bibitem[Bea]{Bea} Beauville, Arnaud,\ \emph{Complex algebraic surfaces.} Translated from the 1978 French original by R. Barlow, with assistance from N. I. Shepherd-Barron and M. Reid. Second edition. London Mathematical Society Student Texts, 34. Cambridge University Press, Cambridge, 1996.
	
	\bibitem[CC]{CC} Chen, Jungkai A.; Chen, Meng,\ \emph{The Noether inequality for Gorenstein minimal 3-folds.} Comm. Anal. Geom. 23 (2015), no. 1, 1–9.
	
	\bibitem[CH]{CH} Chen, Yifan; Hu, Yong,\ \emph{On canonically polarized Gorenstein 3-folds satisfying the Noether equality}. arXiv:1411.2200
	
	\bibitem[EH]{EH} Eisenbud, David; Harris, Joe,\ \emph{On varieties of minimal degree} (a centennial account). Algebraic geometry, Bowdoin, 1985 (Brunswick, Maine, 1985), 3–13, Proc. Sympos. Pure Math., 46, Part 1, Amer. Math. Soc., Providence, RI, 1987.

	\bibitem[F1]{F1} Fujita, Takao,\ \emph{On the structure of polarized varieties with $\Delta$-genera zero.} J. Fac. Sci. Univ. Tokyo Sect. IA Math. 22 (1975), 103-115.

	\bibitem[F2]{F2} Fujita, Takao,\ \emph{On polarized varieties of small $\Delta$-genera.} Tohoku Math. J. (2) 34 (1982), no. 3, 319-341.
	
	\bibitem[F3]{F3}  Fujita, Takao,\ \emph{Classification of projective varieties of $\Delta$-genus one.} Proc. Japan Acad. Ser. A Math. Sci. 58 (1982), no. 3, 113-116.
	
	\bibitem[K]{K} Kobayashi, Masanori,\ \emph{On Noether's inequality for threefolds.} J. Math. Soc. Japan 44 (1992), no. 1, 145-156.
	
	\bibitem[L]{P1} Lazarsfeld, Robert,\ \emph{Positivity in algebraic geometry. I.} Classical setting: line bundles and linear series. Ergebnisse der Mathematik und ihrer Grenzgebiete. 3. Folge. A Series of Modern Surveys in Mathematics [Results in Mathematics and Related Areas. 3rd Series. A Series of Modern Surveys in Mathematics], 48. Springer-Verlag, Berlin, 2004.

\end{thebibliography}
\end{document}